\documentclass[12pt,twoside]{article}
\usepackage{etex}
\usepackage[dvipsnames]{xcolor} % 颜色宏包调用
\usepackage{color}
\usepackage{booktabs}
\usepackage{amsthm}
\usepackage{amsfonts,amssymb,amsxtra,url,float} % 论文初始格式设定 ams 宏包
\allowdisplaybreaks[4] % 允许多行公式强制换页
\usepackage[colorlinks,
  linkcolor=magenta, %
  anchorcolor=Periwinkle,
  citecolor=violet,
  urlcolor=blue
  ]{hyperref} % 超链颜色设定
\usepackage{enumitem}
\usepackage{geometry} 
\usepackage{rotating} % 允许rotating环境旋转段落
\usepackage{lscape} % 允许使用landscape环境进行局部旋转
\usepackage{multirow}
\usepackage{graphicx} % 图片宏包调用, 调用此宏包后将被允许插入png 格式图片
\usepackage{subfigure} % figure环境下的补充宏包, 允许在figure环境下对自图片进行分栏
\usepackage{tikz}
\usetikzlibrary{decorations.pathreplacing,decorations.markings}
 \tikzset{
  on each segment/.style={
    decorate,
    decoration={
      show path construction,
      moveto code={},
      lineto code={
        \path [#1]
        (\tikzinputsegmentfirst) -- (\tikzinputsegmentlast);
      },
      curveto code={
        \path [#1] (\tikzinputsegmentfirst)
        .. controls
        (\tikzinputsegmentsupporta) and (\tikzinputsegmentsupportb)
        ..
        (\tikzinputsegmentlast);
      },
      closepath code={
        \path [#1]
        (\tikzinputsegmentfirst) -- (\tikzinputsegmentlast);
      },
    },
  },
  mid arrow/.style={postaction={decorate,decoration={
        markings,
        mark=at position 0.6 with {\arrow[#1]{stealth}} % 可根据具体情况调整数字0.6
      }}},
}
\usetikzlibrary{arrows}
\usetikzlibrary{trees}

\usetikzlibrary{matrix}
\usetikzlibrary{patterns}
\usetikzlibrary{shadings} % 调用tikz作图指令包
\usepackage[all]{xy} % 交换图箭图宏包调用, 允许在xymatrix环境下进行交换图表编译
\usepackage{dsfont} % 可以调用mathds 指令
\usepackage{cite}
\usepackage{mathrsfs} % mathscr
\numberwithin{figure}{section}
\usepackage{marginnote} % 可使用边注\marginnote{}指令
\usepackage{graphicx} % 可使用\figure 环境
\usepackage{multicol} %用于实现在同一页中实现不同的分栏

\usepackage{enumitem}
\setenumerate[1]{itemsep=0pt,partopsep=0pt,parsep=\parskip,topsep=3pt}
\setitemize[1]{itemsep=0pt,partopsep=0pt,parsep=\parskip,topsep=3pt}
\setdescription{itemsep=0pt,partopsep=0pt,parsep=\parskip,topsep=3pt}
\setlist[itemize]{leftmargin=35pt}
\setlist[enumerate]{leftmargin=35pt}
\newcommand{\checks}[1]{{\color{cyan}{#1}}} % 用于对修改内容的染色

\newtheorem{theorem}{Theorem}[section]
\newtheorem{lemma}[theorem]{Lemma}
\newtheorem{corollary}[theorem]{Corollary}
\newtheorem{main theorem}[theorem]{Main Theorem}

\newtheorem{definition}[theorem]{Definition}

\newtheorem{remark}[theorem]{Remark}
\newtheorem{example}[theorem]{Example}

\newtheorem{question}[theorem]{Question}

\usetikzlibrary{arrows}

\numberwithin{equation}{section}

\def\<{\langle} % 左尖括号指令省略为"\<"
\def\>{\rangle} % 右尖括号指令省略为"\>"
\def\NN{\mathbb{N}} % mathkk双写字体的"N", 自然数集
\newcommand{\Ker}{\operatorname{Ker}}
\newcommand{\Ima}{\operatorname{Im}}

\newcommand{\modcat}{\mathsf{mod}}

\newcommand{\kk}{\mathds{k}} % mathds双写字体的"k", 用于描述域
\newcommand{\Q}{\mathcal{Q}} % mathcal花体的"Q", 用于描述箭图
\def\I{\mathcal{I}}

\newcommand{\e}{\varepsilon}
\newcommand{\Hom}{\mathrm{Hom}} %
\newcommand{\End}{\mathrm{End}} %
\newcommand{\Ext}{\mathrm{Ext}} %
\newcommand{\ol}[1]{\overline{#1}}
\newcommand{\w}[1]{\widetilde{#1}}
\newcommand{\To}[1]{\mathop{-\!\!\!-\!\!\!\longrightarrow}\limits^{#1}}

\def\alg{\mathit{\Lambda}}

\def\emb{\mathfrak{e}}
\def\calR{\mathcal{R}}

\def\heart{{\color{red}\pmb{\heartsuit}}}

\def\kk{\Bbbk}
\def\KK{\mathbb{K}}
\def\=<{\leqslant}
\def\>={\geqslant}
\def\pdim{\mathrm{proj.dim}}

\def\gldim{\mathrm{gl.dim}}
\def\brick{\mathsf{brick}}
\def\sbrick{\mathsf{sbrick}}
\def\rad{\mathrm{rad}}
\def\calA{\mathcal{A}}
\def\calB{\mathcal{B}}

\def\heart{{\color{red}\pmb{\heartsuit}}}

\def\orcid{
\begin{tikzpicture}[baseline=-1mm]
\filldraw[Green!35] (0,0) circle (5pt);
\filldraw[white] (0,0) node{\tiny\textbf{iD}};
\end{tikzpicture}
}

\newcommand{\defines}{\it\color{red}}
\title{\bf Some functors preserving exceptionality
}

\vspace{5mm}

\author{
Dajun Liu$^{\ref{Author1}, \href{https://orcid.org/0009-0001-6073-7587}{\orcid}\ref{orcid1}}$,
Hanpeng Gao$^{\ref{Author2}, \href{https://orcid.org/0000-0001-7002-4153}{\orcid}\ref{orcid2}}$,
Yu-Zhe Liu$^{\ref{Author3}, \href{https://orcid.org/0009-0005-1110-386X}{\orcid}\ref{orcid3},~\ref{CorrespondingAuthor}}$
}
\date{ }
\begin{document}
%=========================================================

%\dedicatory{}
%\subjclass[2020]{\color{red} xxOxx; xxOxx; xxOxx... }
%\keywords{\color{red} keywords; keywords; keywords... }
%\thanks{\color{red} This work is supported by ... }

%=========================================================

%\definecolor{section}{rgb}{0,0,0.5}

%=========================================================
% 正文
%=========================================================

%\marginpar{\tiny\listofchanges}

\maketitle

\begin{enumerate}[label=\textbf{\color{red}\arabic*}] \footnotesize
  \item
    \begin{center}
      Anhui Polytechnic University, Wuhu 241000 Anhui, China;

      E-mail: \url{liudajun@ahpu.edu.cn}
    \end{center} \label{Author1} \vspace{2mm}

  \item
    \begin{center}
      School of Mathematical Sciences, Anhui University,
      Hefei 230601, Anhui, China;

      E-mail: \url{hpgao@ahu.edu.cn}
    \end{center} \label{Author2} \vspace{2mm}

  \item
    \begin{center}
      School of Mathematics and Statistics, Guizhou University,
      Guiyang 550025, Guizhou, China;

      E-mail:  \url{liuyz@gzu.edu.cn} / \url{yzliu3@163.com}
    \end{center} \label{Author3}
\end{enumerate}
\hspace{1mm}
\begin{enumerate}[label=\textbf{\color{red}$\dag$}]
  \item \footnotesize
    \begin{center}
      Corresponding author
    \end{center} \label{CorrespondingAuthor}
\end{enumerate}
\hspace{1mm}
\begin{enumerate}[leftmargin=6.5cm] \footnotesize
  \item[\orcid]
      ORCID: \href{https://orcid.org/0009-0001-6073-7587}{0009-0001-6073-7587}
      \label{orcid1}
  \item[\orcid]
      ORCID: \href{https://orcid.org/0000-0001-7002-4153}{0000-0001-7002-4153}
      \label{orcid2}
  \item[\orcid]
      ORCID: \href{https://orcid.org/0009-0005-1110-386X}{0009-0005-1110-386X}
      \label{orcid3}
\end{enumerate}

\hspace{2mm}

%\begin{abstract}
%We constructed some tensor functors that send each exceptional sequence in a module category to another exceptional sequence in another module category by using split extensions and recollements.
%\end{abstract}

\begin{enumerate}
  \item[] \footnotesize
  \textbf{Abstract}:
    We constructed some tensor functors that send each exceptional sequence in a module category to another exceptional sequence in another module category by using split extensions and recollements.
\end{enumerate}

\begin{enumerate}
  \item[] \footnotesize
     %18G05. % Projectives and injectives (category-theoretic aspects)
     \label{2020MSC}
\end{enumerate}

\begin{enumerate}
  \item[] \footnotesize
     \label{Keywords}
\end{enumerate}

%% 目录设置
%
%\setcounter{tocdepth}{3}
%\setcounter{secnumdepth}{3}
%\tableofcontents

%=========================================================

\section{Introduction}

The notion of exceptional sequences was introduced by Gorodentsev and Rudakov \cite{GR1987} in the study of exceptional vector bundles over $\mathbb{P}^2$.
Later, this concept was adapted to the quiver representation framework established by Crawley--Boevey \cite{CB1993}
and Ringel \cite{R1994}.
In recent years, topics on exceptional sequences have remained highly popular. In triangulated categories,
Aihara and Iyama \cite{AI2012} show that there is a bijection between full exceptional sequences and silting objects.
For any finite-dimensional algebra, Buan and Marsh \cite{BM2021} introduce
the notions of $\tau$-exceptional and signed $\tau$-exceptional sequences
and prove that for a fixed algebra of rank $n$, and for any positive integer $t\leqslant n$,
there is a bijection between the set of signed $\tau$-exceptional sequences of length $t$,
and ordered support $\tau$-rigid objects with $t$ indecomposable direct summands.
The notion of weak exceptional sequences was introduced by Sen \cite{S2021},
which can be viewed as another modification of the standard case, different from the works of Buan and Marsh \cite{BM2021}.
Westin and Thuresson \cite{WT2022} provide a classification of generalized tilting modules and full exceptional sequences for a certain family of quasi-hereditary algebras, specifically dual extension algebras.
The theory of tilting over split-by-nilpotent extensions
and recollements have garnered significant attention. For example, Assem and Marmaridis \cite{AM1998},
as well as Assem and Zacharia \cite{AZ2003}, along with Liu and Wei \cite{LW2022},
and Suarez \cite{Sua2020} have investigated numerous interesting results concerning the relationship between tilting modules,
tilting pairs and support $\tau$-tilting modules in $\modcat(A)$ and induced objects in $\modcat(R)$.

A Recollements of Abelian categories, denoted by $\calR(\mathcal{A},\mathcal{B},\mathcal{C})$, is a diagram
\begin{align} \label{recoll diagram}
\xymatrix@C=2cm{
  \mathcal{A}
  \ar[r]^{i_*}
&  \mathcal{B}
  \ar[r]^{j^*}
  \ar@/_1.5pc/[l]_{i^*}
  \ar@/^1.5pc/[l]^{i^!}
&  \mathcal{C},
  \ar@/_1.5pc/[l]_{j_!}
  \ar@/^1.5pc/[l]^{j_*}
}
\end{align}
such that
\begin{enumerate}[label=(R\arabic*), leftmargin=1.6cm]
  \item $(i^*, i_*)$, $(i_*, i^!)$, $(j_!, j^*)$, $(j^*, j_*)$ are adjoint pairs; \label{R1}
  \item the functors $i_*$, $j_!$, and $j_*$ are fully faithful; \label{R2}
  \item $\mathrm{Im}(i_*)=\mathrm{Ker}(j^*)$. \label{R3}
\end{enumerate}
Recollements of Abelian and triangulated categories were introduced by Beilinson, Bernstein, and Deligne \cite{BBDG1982} in connection with derived categories of sheaves on topological spaces,
with the idea that one triangulated category may be ``glued together'' from two others.
This concept plays an important role in the representation theory of algebras, and many algebraists and authors have researched on the recollements,
for example, weighted projective line \cite{Chen2012},
Abelian and triangulated categories \cite{zhang2017subcate, FP2004, PV2019, ZhangY2023},
Homological theory \cite{Psa2014, QinHan2016, Zhang2017, ZZ2020JAA},
Gorenstein-Homological theory \cite{Zhang2013Gorenstein, ZZ2020},
and tilting theory \cite{M2018,MZ2020, MXZ2022}, etc.
% for example, \cite[ect]{Chen2012, Zhang2013Gorenstein, QinHan2016, Qin2020CMAuslander, ZZ2020, ZZ2020JAA, ZhangY2023}.
% The theory of tilting with respect to a recollement of triangulated or Abelian categories has been investigated widely; see for example, \cite[etc]{M2018,MXZ2022,MZ2020}.
Psaroudakis provided an important recollement of module categories which is the form $\calR(\modcat(A/A\e A), A, \e A\e)$
in \cite[Example 2.7]{Psa2014} (see Example \ref{exp:recollement} in this paper),
where $\e$ is an idempotent of $A$, such that $i^*$ and $j_!$ are two tensor functors
\begin{align*}
  i^* = - \otimes_A{_A(A/A\e A)}:
\ & \modcat(A) \to \modcat(A/A\e A) \\
\text{and~}  j_! = -\otimes_{\e A\e} \e A:
\ & \modcat(\e A \e) \to \modcat (A)
\end{align*}
of module categories.

On the other hand, we know that split extensions $\xi: R\to A$ of an algebra $A$ by a nilpotent bimodule,
introduced by Assem and Zacharia in \cite{AZ2003}, is also an important concept providing tensor functors
\begin{align*}
  -\otimes_A {_AR_R}:
\ & \modcat(A) \to \modcat(R) \\
\text{and~}  -\otimes_R~{_RA_A}:
\ & \modcat(R) \to \modcat (A)
\end{align*}
of module categories. Then the following question is natural:

\begin{question}
Let $F$ be a tensor functor of module categories from $\modcat(A)$ to $\modcat(B)$.
Then for each exceptional sequence $(M_{1}, \ldots, M_{r})$ in $\modcat(A)$,
when $F$ satisfies what conditions, $(F(M_{1}), \ldots, F(M_{r}))$ is an exceptional sequence?
\end{question}

We want to answer the above question by using recollements and split extensions.
Then we obtain two main results of this paper.

\begin{theorem}
Let $\xi:R\to A$ be a split extension of $A$ by a nilpotent bimodule $Q$
and $\calR(A', A, A'')$ be a recollement of module categories.
\begin{itemize}
\item[{\rm(1)}] {\rm(}Theorem \ref{Th2.1}{\rm)}
For any  exceptional sequence $(M_1,\ldots,M_r)$ in $\modcat(A)$, if:
\begin{itemize}
  \item $R={_AR}$ is a projective left $A$-module;
  \item $\Hom_{A}(M_k, M_k\otimes_{A}Q) =0$ $(1\=< k \=< r)$
     and $\Hom_{A}(M_j, M_i\otimes_{A}Q)=0$ $(1\=< i< j\=< r)$ hold;
  \item $\Ext^{n}_{A}(M_k, M_k\otimes_{A}Q) =0$ $(1\=< k \=< r)$
     and $\Ext^{n}_{A}(M_j, M_i\otimes_{A}Q)=0$ $(1\=< i< j\=< r)$ hold,
\end{itemize}
then $(M_1\otimes_{A}R, \ldots, M_r\otimes_{A}R)$ is an exceptional sequence in $\modcat(R)$.

\item[{\rm(2)}] {\rm(}Theorem \ref{Th3.18}{\rm)}
For any exceptional sequence $(X_1, \ldots, X_t)$ in $\modcat(A')$
and any exceptional sequence $(Y_1, \ldots, Y_u)$ in $\modcat(A'')$,
if $i^{*}$ and $i^{!}$ are exact functors, then $(i_*(X_1)$, $\ldots$, $i_*(X_t))$
and $(j_!(Y_1)$, $\ldots$, $j_!(Y_u))$ are exceptional sequences in $\modcat(\alg)$.
\end{itemize}
\end{theorem}

% Inspired by the above results, especially noticing the close relationship between exceptional sequences and tilting theory, we naturally consider the mutual transformation between exceptional sequences in $A$-mod and those in $R$-mod under the action of functors $-\otimes_{A}R$ and $-\otimes_{R}A$. Moreover, the gluing properties of exceptional sequences under recollements are obtained.

\section{Preliminaries} \label{sect:pre}

In this section, we will give some terminologies and some preliminary results. Throughout this article, $\KK$ denotes an algebraically closed field, and $n$ represents a positive integer. Let $A$ be a finite-dimensional $\KK$-algebra. We consistently assume that the category $A$-mod consisting of finitely generated right $A$-modules. Denote the projective dimensions of module $M_A$ by $\pdim(M_A)$.

\subsection{Exceptional Sequences} \label{subsect:gent}
Recall, as seen in \cite{GR1987}, that an indecomposable left (or right) $A$-module $M$ is called {\defines exceptional}  provided that
\begin{enumerate}[label=(E\arabic*), leftmargin=1.5cm]
  \item $\End_{A}M$ is isomorphic to $\KK$; \label{E1}
  \item $\Ext_{A}^{n}(M, M) =0$, for any positive integer $n$.  \label{E2}
\end{enumerate}
A module pair $(M_{i}, M_{j})$ of exceptional modules in $\modcat(A)$ is called an {\defines exceptional pair} if
\begin{enumerate}[label=(E\arabic*$'$), leftmargin=1.6cm]
  \item Hom$(M_{j}, M_{i})=0$;  \label{E1'}
  \item Ext$_{A}^{n}(M_{j}, M_{i}) =0$, for any positive integer $n$.  \label{E2'}
\end{enumerate}
Furthermore, a module sequence $(M_{1}, M_{2}, \ldots, M_{r})$ consisting of exceptional modules is deemed an {\defines exceptional sequence}
if every pair $(M_{i}, M_{j})$ with $1 \=< i <j \=< r$ satisfies the above criteria for being exceptional. It is said to be {\defines complete} if $r = r_A$, the rank of the Grothendieck group $K(A)$.

Recall that the following definitions are given in \cite{Asai2020}.
\begin{itemize}
\item[(1)] $S\in \modcat(A)$ is called a {\defines brick} if $\Hom_{A}(S, S)$ is isomorphic to $\KK$.
  The set of isoclasses of bricks in $\modcat(A)$ will be denoted by $\brick(A)$.
\item[(2)] A subset $\mathcal{S}\subseteq$ $\brick(A)$ is called a {\defines semibrick} if $\mathrm{Hom}_{A}(S_1,S_2)=0$ for any $S_1\neq S_2 \in \mathcal{S}$. The set of semibricks in $A\modcat$ will be denoted by $\sbrick(A)$.
\end{itemize}

It is evident that every simple module is a brick, and a set of isoclasses of simple modules forms a semibrick. However, while every exceptional module is indeed a brick, a sequence
$(S_1, S_2, \cdots, S_r)$ of isoclasses of simple modules does not necessarily constitute an exceptional sequence. As demonstrated in \cite{Gao2021}, there exists a bijection between
$\tau$-tilting modules and sincere left finite semibricks, facilitating the construction of (sincere) semibricks over split-by-nilpotent extensions.
% These findings serve as additional motivations for this dissertation. The following theorems are our main results:

\subsection{Split-by-nilpotent extensions}%}
Let $A$, $R$ be two finite dimensional algebras.

\begin{definition}[\!\!{\cite[Definition 1.1]{AZ2003}}] \rm \label{def:split ext}
We say that $R$ is a {\defines split extension} of $A$ by the nilpotent $(A,A)$-bimodule $Q$,
or {\defines simply a split-by-nilpotent extension} if there exists a split surjective algebra morphism $\xi:R\to A$ whose kernel $Q := \Ker(\xi) = \{ r\in R \mid \xi(r)=0 \}$ is contained in the radical $\rad R$ of $R$.
\end{definition}

\begin{remark} \label{rmk-2.5} \rm
Let $R$ be a split-by-nilpotent extension of $A$ by the nilpotent bimodule $Q$. Clearly, the short exact sequence of $(A,A)$-bimodules
\begin{center}
$0\to ~_{A}Q_{A}\to ~_{A}R_{A}\to ~_{A}A_{A}\to 0$
\end{center}
splits. Therefore, there exists an isomorphism $_{A}R_{A} \cong A\oplus {_{A}Q_{A}}$.
The module categories $\modcat(A)$ and $\modcat(R)$ are related by the following four functors
\begin{align*}
& -\otimes_A R: \modcat(A) \to \modcat(R), \\
& -\otimes_R A: \modcat(R) \to \modcat(A), \\
& \Hom_A({_AR},-): \modcat(A) \to \modcat(R), \\
& \Hom_R({_RA},-): \modcat(R) \to \modcat(A).
\end{align*}
Moreover, we have
\begin{align*}
& (-\otimes_A R) \otimes_R A \simeq 1_{\modcat(A)},  \\
& \Hom_R(A_R, \Hom_A(R_A,-)) \simeq 1_{\modcat(R)}.
\end{align*}
\end{remark}

\subsection{Recollements of Abelian categories}
%We recall the notion of recollements of Abelian categories.
%\begin{definition}[\!\!{\cite{FP2004}}]
%\label{def-2.1} \rm
%A recollement, denoted by $\calR(\mathcal{A},\mathcal{B},\mathcal{C})$, of Abelian categories is a diagram
%\begin{align} \label{recoll diagram}
%\xymatrix@C=2cm{
%  \mathcal{A}
%  \ar[r]^{i_*}
%&  \mathcal{B}
%  \ar[r]^{j^*}
%  \ar@/_1.5pc/[l]_{i^*}
%  \ar@/^1.5pc/[l]^{i^!}
%&  \mathcal{C},
%  \ar@/_1.5pc/[l]_{j_!}
%  \ar@/^1.5pc/[l]^{j_*}
%}
%\end{align}
%of Abelian categories and additive functors such that
%\begin{enumerate}[label=(R\arabic*), leftmargin=1.6cm]
%  \item $(i^*, i_*)$, $(i_*, i^!)$, $(j_!, j^*)$, $(j^*, j_*)$ are adjoint pairs; \label{R1}
%  \item the functors $i_*$, $j_!$, and $j_*$ are fully faithful; \label{R2}
%  \item $\mathrm{Im}(i_*)=\mathrm{Ker}(j^*)$. \label{R3}
%\end{enumerate}
%\end{definition}

We list some properties of recollements (see \cite{M2018, Psa2014, PV2019}) in this section, which will be used in the sequel.

\begin{lemma}\label{2.6} Let $(\mathcal{A}, \mathcal{B}, \mathcal{C})$ be a recollement of abelian categories. Then we have
\begin{itemize}
  \item[{\rm(1)}] $i^* j_{!}=0=i^{!} j_*$;
  \item[{\rm(2)}] the functors $i_*$ and $j^*$ are exact, $i^*$ and $j_{!}$are right exact, and
             $i^{!}$and $j_*$ are left exact;
  \item[{\rm(3)}]  all the natural transformations

             $$i^* i_* \rightarrow 1_{\mathcal{A}},  1_{\mathcal{A}} \rightarrow i^{!} i_*,  1_{\mathcal{C}} \rightarrow j^* j_{!},  j^* j_* \rightarrow 1_{\mathcal{C}}
              $$
              are natural isomorphisms;

  \item[{\rm(4)}]$\text { if } i^* \text { is exact, then } i^{!} j_{!}=0 \text {, and if } i^{!} \text {is exact, then } i^* j_*=0 \text {. }$
\end{itemize}
\end{lemma}

The following lemma is widely studied, which plays a crucial role in the sequel.

\begin{lemma}[{\!\!\cite[Example 2.7]{Psa2014}}] \label{exp:recollement}
Let $\e$ be an idempotent of an algebra $A$.
Denote by $\ol{A}$ the quotient $A/A\e A$ of $A$ and $\w{A}$ the algebra $\e A\e$.
Then we have a recollement of module categories:
\begin{align}\label{recollement}
\xymatrix@C=2cm{
  \modcat(\ol{A})
  \ar[r]^{\emb}_{\mathrm{embedding}}
& \modcat(A)
  \ar[r]^{(-)\e}_{\text{retraction}}
  \ar@/_1.5pc/[l]_{ - \otimes_A{_A\ol{A}}}
  \ar@/^1.5pc/[l]^{\Hom_A({\ol{A}_A}, -)}
& \modcat(\w{A}),
  \ar@/_1.5pc/[l]_{-\otimes_{\w{A}} \e A }
  \ar@/^1.5pc/[l]^{\Hom_{\w{A}}(\e A,-)}
}
\end{align}
where $\emb$ is an embedding functor.
\end{lemma}

\section{Functors preserving exceptionality}

\subsection{Projective split extensions and functor $-\otimes_A R_R$}

A split extension $\xi:R\to A$ of $A$ by the nilpotent $(A,A)$-bimodule $Q$
is said to be {\defines projective} if $R = {_AR}$, as a left $A$-module, is projective.
In this section, we show that $-\otimes_A R$ preserves the exceptionality of exceptional sequence, see Theorem \ref{Th2.1}.

\begin{lemma} \label{Lem3.1}
Let $\calA$, $\calB$ be two Abelian categories with enough projective objects,
and $F: \calA \to \calB$, $G: \calB \to \calA$ be two functors.
If $(F,G)$ is an adjoint pair and $F$ is an exact functor preserving projective objects, then
\begin{center}
  $\Ext^{n}_{\calB}(F(M), F(N)) \cong \Ext^{n}_{\calA}(M, GF(N))$
\end{center}
holds for each objects $M$, $N$ in $\calA$.
\end{lemma}

\begin{proof}
Assume that the projective resolution of ${M}$ is
\[ \cdots \To{} P_1 \To{} P_0 \To{} M \To{} 0.  \eqno{\text{$\spadesuit$}} \]
Then we obtain that
\[ \cdots \To{} F(P_1) \To{} F(P_0) \To{} F(M) \To{} 0 \eqno{\text{$\clubsuit$}} \]
is a projective resolution of $F(M)$, since $F$ is an exact functor and preserves projective modules.
By applying the functor $\Hom_{\calB}(-, GF(N))$ to sequence $\spadesuit$ and
applying the functor $\Hom_{\calB}(-, F(N))$ to sequence $\clubsuit$,
we obtain that the following diagram
$$
\xymatrix{
  0 \ar[r]
& \Hom_{\calB}(F(M), F(N))   \ar[r] \ar[d]^{\cong}
& \Hom_{\calB}(F(P_0), F(N)) \ar[r] \ar[d]^{\cong}
& \Hom_{\calB}(F(P_1), F(N)) \ar[d]^-{\cong} \ar[r]
& \cdots \\
  0 \ar[r]
& \Hom_{\calA}(M, GF(N))   \ar[r]
& \Hom_{\calA}(P_0, GF(N)) \ar[r]
& \Hom_{\calA}(P_1, GF(N))\ar[r]&\cdots
}
$$
commutes since $(F,G)$ is an adjoint pair.
Thus, $\Ext^{n}_{\calB}(F(M), F(N))\cong \Ext^{n}_{\calA}(M, GF(N))$
holds for all $n\in\NN$.
\end{proof}

\begin{corollary} \rm\label{Cor3.2}
Let $\xi:R\to A$ be a projective split extension of a finite-dimensional algebra $A$ by the nilpotent $(A,A)$-bimodule $Q$.
Then
\begin{center}
   $\Ext^{n}_{R}(M\otimes_{A}R , N\otimes_{A}R)\cong \Ext^{n}_{A}(M_A, \Hom_R({_AR_R}, N\otimes_{A}R))$
\end{center}
holds for all $n\in\NN$.
\end{corollary}

\begin{proof}
First of all, since ${_A}R$ is projective, then there exists a family of idempotents $(e_i)_{1\=< i \=< t}$
such that \[{_AR} \cong A \oplus {_AQ} \cong A \oplus \bigoplus\limits_{i=1}^t A e_i,\]
where $Q \cong \bigoplus\limits_{i=1}^t A e_i$.
Then, for each projective right $A$-module $eA$, the following formula
\begin{align*}
  eA \otimes_A R
& \cong  (eA \otimes_A A) \oplus \Big(eA \otimes_A \Big(\bigoplus\limits_{i=1}^t A e_i\Big) \Big) \\
& \cong eA \oplus \Big(\bigoplus\limits_{i=1}^t (eA \otimes_A Ae_i)\Big) \\
& \cong eA \oplus \bigoplus\limits_{i=1}^t eAe_i \\
& = e\Big(A\oplus \bigoplus\limits_{i=1}^t Ae_i\Big) \\
& \cong e R
\end{align*}
admits that $-\otimes_A R$ preserves projective modules.
Note that $(-\otimes_{A}R, \Hom_{R}(R, -))$ is an adjoint pair,
and since $R={_AR}$ is projective we have $-\otimes_{A}R$ is an exact.
Then, by Lemma \ref{Lem3.1}, we obtain that
\begin{center}
$ \Ext^{n}_{R}(M\otimes_{A}R , N\otimes_{A}R)
  \cong \Ext^{n}_{A}(M, \Hom_{R}({_AR_R}, N\otimes_{A}{_AR_R})) $
\end{center}
holds for all $n\in \NN$.
\end{proof}

Example \ref{Ex3.3} given in Section \ref{Sect:Examp} illustrates that the condition ``${_A}R$ is projective'' is indispensable, yet it is not uncommon. We provide an instance in Example \ref{Ex3.4} for it.

\begin{lemma} \label{Lem3.5}
Let $\xi:R\to A$ be a split extension of $A$ by the nilpotent bimodule $Q$
and $M_A$ be an exceptional module in $\modcat(A)$. If
\begin{enumerate}[label={\rm(\arabic*)} ]
  \item $\xi$ is projective; \label{Lem3.5(1)}
  \item $M_A$ satisfies $\Hom_{A}(M, M\otimes_{A}Q)=0$; \label{Lem3.5(2)}
  \item and $M_A$ satisfies $\Ext_{A}^{n}(M, M\otimes_{A}Q) =0$, \label{Lem3.5(3)}
\end{enumerate}
then $M\otimes_{A}R$, as a right $R$-module, is exceptional.
\end{lemma}

\begin{proof}
First of all, we have the following isomorphism
\begin{equation*}
\begin{split}
\Hom_{R}(M\otimes_{A}R, M\otimes_{A}R)
&\cong \Hom_{A}(M_{A}, \Hom_{R}(R, M\otimes_{A}R)) \\
&\cong \Hom_{A}(M_{A}, M\otimes_{A}R)  \\
&\cong \Hom_{A}(M_{A}, M\otimes_{A}(A\oplus Q))  \\
&\cong \Hom_{A}(M_{A}, M_{A})\oplus \Hom_{A}(M_{A}, M\otimes_{A}Q).
\end{split}
\end{equation*}
Since $M_A$ is an exceptional module in $\modcat(A)$, we have $\Hom_{A}(M_{A}, M_{A}) = \KK$ by \ref{E1}.
Moreover, by the condition \ref{Lem3.5(2)} of this lemma, $\Hom_{A}(M_{A}, M\otimes_{A}Q) = 0$. Then we obtain
\[ \Hom_{R}(M\otimes_{A}R, M\otimes_{A}R) = \KK, \]
i.e., the condition \ref{E1} holds for $M \otimes_A R$.

Second, by the condition \ref{Lem3.5(1)} of this lemma and Corollary~\ref{Cor3.2}, we have the following isomorphism
\begin{equation*}
\begin{split}
\Ext^{n}_{R}(M\otimes_{A}R, M\otimes_{A}R)
&\cong \Ext^{n}_{A}(M_{A}, \Hom_{R}(R, M\otimes_{A}R)) \\
&\cong \Ext^{n}_{A}(M_{A}, M\otimes_{A}R)  \\
&\cong \Ext^{n}_{A}(M_{A}, M\otimes_{A}(A\oplus Q))  \\
&\cong \Ext^{n}_{A}(M, M)\oplus \Ext^{n}_{A}(M_{A}, M\otimes_{A}Q),
\end{split}
\end{equation*}
By the condition \ref{Lem3.5(3)} of this lemma and \ref{E2}, we have
 \[\Ext^{n}_{R}(M\otimes_{A}R, M\otimes_{A}R) = 0, \]
i.e., the condition \ref{E2} holds for $M \otimes_A R$.
Thus, $M\otimes_{A}R$ is an exceptional module in $\modcat(R)$.
\end{proof}

\begin{lemma} \label{Lem3.6}
Let $\xi:R\to A$ be a split extension of $A$ by the nilpotent bimodule $Q$ and
$M_1$ and $M_2$ be two modules in $\modcat(A)$.  If
\begin{enumerate}[label={\rm(\arabic*)}]
  \item $(M_1, M_2)$ is an exceptional pair in $\modcat(A)$; \label{Lem3.6(1)}
  \item $\xi$ is projective; \label{Lem3.6(2)}
  \item $\Hom_{A}(M_k, M_k\otimes_{A}Q)=0$ $(k\in\{1,2\})$ and $\Hom_{A}(M_2, M_1\otimes_{A}Q)=0$ hold; \label{Lem3.6(3)}
  \item $\Ext^{n}_{A}(M_k, M_k\otimes_{A}Q) =0$ $(k\in\{1,2\})$ and $\Ext_{A}^{n}(M_2, M_1\otimes_{A}Q) =0$ hold for all $n\in \NN$,\label{Lem3.6(4)}
\end{enumerate}
Then $(M_1\otimes_{A}R, M_2\otimes_{A}R)$ is an exceptional pair in $\modcat(R)$.
\end{lemma}

\begin{proof}
First of all, by Lemma \ref{Lem3.5}, the conditions \ref{Lem3.6(2)},
\ref{Lem3.6(3)} ($\Hom_{A}(M_k, M_k\otimes_{A}Q)=0$ $(k\in\{1,2\})$) and
\ref{Lem3.6(4)} ($\Ext^{n}_{A}(M_k, M_k\otimes_{A}Q) =0$ $(k\in\{1,2\})$)
admit that $M_1\otimes_{A}Q$ and $M_2\otimes_{A}Q$ are exceptional modules.

Second, by using adjoint isomorphism, we have
\begin{equation*}
\begin{split}
\Hom_{R}(M_2\otimes_{A}R, M_1\otimes_{A}R)
&\cong \Hom_{A}(M_2, \Hom_{R}(R, M_1\otimes_{A}R))\\
&\cong \Hom_{A}(M_2, M_1\otimes_{A}R)\\
&  =   \Hom_{A}(M_2, M_1\otimes_{A}(A\oplus Q))\\
&\cong \Hom_{A}(M_2, M_1)\oplus \Hom_{A}(M_2, M_1\otimes_{A}Q)  \\
\end{split}
\end{equation*}
By the condition \ref{Lem3.6(1)} of this lemma and the condition \ref{E1'},
we have $\Hom_{A}(M_2, M_1) = 0$. By the condition \ref{Lem3.6(3)},
we have $\Hom_{A}(M_2, M_1\otimes_{A}Q) = 0$.
Thus, $\Hom_{R}(M_2\otimes_{A}R, M_1\otimes_{A}R) = 0$, i.e.,
it follows that \ref{E1'} holds for $(M_1\otimes_{A}R, M_2\otimes_{A}R)$.

Third, by Corollary \ref{Cor3.2}, we have
\begin{equation*}
\begin{split}
\Ext^{n}_{R}(M_2\otimes_{A}R, M_1\otimes_{A}R)
&\cong \Ext^{n}_{A}(M_2, \Hom_{R}(R, M_1\otimes_{A}R)) \\
&\cong \Ext^{n}_{A}(M_2, M_1\otimes_{A}R) \\
&\cong \Ext^{n}_{A}(M_2, M_1\otimes_{A}(A\oplus Q)) \\
&\cong \Ext^{n}_{A}(M_2, M_1)\oplus \Ext^{n}_{A}(M_2, M_1\otimes_{A}Q)  \\
\end{split}
\end{equation*}
By \ref{E2'} and the condition \ref{Lem3.6(4)} of this lemma, we have $\Ext^{n}_{R}(M_2\otimes_{A}R, M_1\otimes_{A}R) = 0$,
i.e., \ref{E2'} holds for $(M_1\otimes_{A}R, M_2\otimes_{A}R)$.
The above arguments follows that $(M_1\otimes_{A}R, M_2\otimes_{A}R)$ is an exceptional pair in $\modcat(R)$.
\end{proof}

Now, we show the first main result of this paper.

\begin{theorem} \label{Th2.1} % main 1
Let $\xi:R\to A$ be a split extension of $A$ by the nilpotent bimodule $Q$,
and $M_1$, $M_2$, $\ldots$, $M_r$ be right $A$-modules. If:
\begin{enumerate}[label={\rm(\arabic*)}]
  \item $(M_1, M_2, \ldots, M_r)$ is an exceptional sequence; \label{Th2.1(1)}
  \item $\xi$ is projective; \label{Th2.1(2)}
  \item $\Hom_{A}(M_k, M_k\otimes_{A}Q) =0$ $(1\=< k \=< r)$
     and $\Hom_{A}(M_j, M_i\otimes_{A}Q)=0$ $(1\=< i< j\=< r)$ hold; \label{Th2.1(3)}
  \item $\Ext^{n}_{A}(M_k, M_k\otimes_{A}Q) =0$ $(1\=< k \=< r)$
     and $\Ext^{n}_{A}(M_j, M_i\otimes_{A}Q)=0$ $(1\=< i< j\=< r)$ hold, \label{Th2.1(4)}
\end{enumerate}
then $(M_1\otimes_{A}R, M_2\otimes_{A}R, \ldots, M_r\otimes_{A}R)$ is an exceptional sequence in $\modcat(R)$.
In addition, if $(M_1, M_2, \ldots, M_r)$ is complete, then so is $(M_1\otimes_{A}R, M_2\otimes_{A}R, \ldots, M_r\otimes_{A}R)$.
\end{theorem}

\begin{proof}
All modules $M_k\otimes_{A}Q$ ($1\=< k \=< r$) are exceptional by Lemma \ref{Lem3.5} and the conditions \ref{Th2.1(2)}, \ref{Th2.1(3)} and \ref{Th2.1(4)} of this theorem.
Moreover, for each $1\=< i< j\=< r$, \ref{E1'} and \ref{E2'} hold for $(M_i\otimes_{A}R, M_j\otimes_{A}R)$ by Lemma \ref{Lem3.6}.
Thus, $(M_1\otimes_{A}R, M_2\otimes_{A}R, \ldots, M_r\otimes_{A}R)$ is an exceptional sequence.
\end{proof}

% \subsection{Functor $-\otimes_A A_R$}

\subsection{Functors $i_*$ and $j_!$}

%\begin{lemma}[{Ma--Xie--Zhao}] \label{Lem3.14}
%Let $\calR(\alg',\alg,\alg'')$ be a recollement as shown in (\ref{recoll diagram}).
%If $i^*$ is an exact functor, then $i_*$, $j^*$, and $j_!$ are exact functors preserving projective modules.
%\end{lemma}
%
%\begin{proof}
%\cite[Proposition 2.6]{MXZ2022} has showed that $i_*$ and $j^*$ are exact functors preserving projective modules.
%Since $\modcat(\alg'')$ has enough projective objects,
%then $j_!$ preserves projective modules by \cite[Proposition 2.5 (2)]{MXZ2022}.
%Moreover, by using \cite[Lemma 2.4 (1)]{MXZ2022}, we obtain that $j_!$ is exact since $i^*$ is exact,
%then this completes the proof of this lemma.
%\end{proof}

\begin{lemma}[{Ma--Xie--Zhao \cite{MXZ2022}}] \label{Lem3.14}
Let $\calR(\alg',\alg,\alg'')$ be a recollement as shown in (\ref{recoll diagram}).
If $i^*$ and $i^!$ are exact functors, then $i_*$ and $j_!$ are exact functors preserving projective modules.
\end{lemma}

\begin{proof}
\cite[Proposition 2.6 (2)]{MXZ2022} has showed that $i_*$  functor preserving projective modules whenever $i^!$ is exact. In addition, $i_*$ is an exact functor by \cite[Proposition 2.2 (2)]{MXZ2022}.
Since $\modcat(\alg'')$ has enough projective objects,
then $j_!$ preserves projective modules by \cite[Proposition 2.5 (2)]{MXZ2022}.
Moreover, by using \cite[Lemma 2.4 (1)]{MXZ2022}, we obtain that $j_!$ is exact since $i^*$ is exact,
then this completes the proof of this lemma.
\end{proof}

\begin{lemma} \label{lem3.15}
Let $\calR(\alg',\alg,\alg'')$ be a recollement of module categories.
If  $i^*$ and $i^!$ are exact functors, then for any $M\in\modcat(\alg')$, $N\in\modcat(\alg'')$, and $n\>=1$, we have
\begin{enumerate}[label={\rm(\arabic*)}]
  \item $\Ext_{\alg}^n(i_*(M), i_*(M)) \cong \Ext_{\alg' }^n(M, M) = 0$;
  \item $\Ext_{\alg}^n(j_!(N), j_!(N)) \cong \Ext_{\alg''}^n(N, N) = 0$.
\end{enumerate}
\end{lemma}

\begin{proof}
By the definition of recollement $\calR(\alg',\alg,\alg'')$,
we have that $(i_*, i^!)$ is an adjoint pair, see \ref{R1},
and since $i^*$ is exact, we obtain that $i_*$ is an exact functor preserving projective modules by Lemma \ref{Lem3.14}.
Then, by Lemma \ref{Lem3.1},
\begin{center}
  $\Ext_{\alg}^n(i_*(M), i_*(M)) \cong \Ext_{\alg' }^n(M, i^!(i_*(M))) = 0$
\end{center}
holds for all $n\>=1$.
Notice that $i^!i_* = 1_{\modcat(\alg')}$ (see Lemma \ref{2.6} (3)), then the above follows that
\begin{center}
  $\Ext_{\alg}^n(i_*(M), i_*(M)) \cong \Ext_{\alg' }^n(M, M) = 0$
\end{center}
holds for all $n\>=1$.
One can check that $\Ext_{\alg}^n(j_!(N), j_!(N)) \cong \Ext_{\alg''}^n(N, N) = 0$ holds for all $n\>=1$ by a similar way.
\end{proof}

\begin{lemma} \label{Lem3.16}
Let $M$ and $N$ be exceptional modules in $\modcat(\alg')$ and $\modcat(\alg'')$, respectively.
If $i^*$ and $i^!$ are exact functors, then $i_*(M)$ and $j_!(N)$ are exceptional right $\alg$-modules in $\modcat(\alg)$.
\end{lemma}

\begin{proof}
We need show that \ref{E1} and \ref{E2} hold for $i_*(M)$ and $j_!(N)$ in this proof.

First of all, by \ref{R2}, $i_*$ is fully faithful, then we have the following isomorphism
\begin{align*}
  \End_{\alg}(i_*(M)) \cong \End_{\alg'}(M) \cong \KK
\end{align*}
since $M$ is exceptional in $\modcat(\alg')$. Thus, \ref{E1} holds for $i_*(M)$.
Similarly, one can check that \ref{E1} holds for $j_!(N)$ by using $j_!$ being fully faithful
and $N$ being exceptional in $\modcat(\alg'')$.

Second, for any $n \>= 1$, we have
\begin{center}
  $\Ext_{\alg' }^n(M, M) \cong \Ext_{\alg}^n(i_*(M),i_*(M)) = 0$
\end{center}
by Lemma \ref{lem3.15}, then \ref{E2} hold for $i_*(M)$.
One can check that \ref{E2} holds for $j_!(N)$ by a similar way.
\end{proof}

\begin{lemma} \label{Lem3.17}
Let $(M_1, M_2)$ and $(N_1, N_2)$ be exceptional pairs in $\modcat(\alg')$ and $\modcat(\alg'')$, respectively.
If $i^*$ and $i^!$ are exact functors, then $(i_*(M_1), i_*(M_2))$ and $(j_!(N_1), j_!(N_2))$ are exceptional pairs in $\modcat(\alg)$.
\end{lemma}

\begin{proof}
since modules $M_1$ and $M_2$ are exceptional in $\modcat(\alg')$ and $i^{*}$ and $i^{!}$ are exact,
then $i_*(M_1)$ and $i_*(M_2)$ are exceptional in $\modcat(\alg)$ by Lemma \ref{Lem3.16}.
By Lemma \ref{2.6} (3), $i_*$ is fully faithful, then
\[ \Hom_{\alg}(i_*(M_2), i_*(M_1)) \cong \Hom_{\alg'}(M_2, M_1) = 0, \]
i.e, \ref{E1'} holds for $(i_*(M_1), i_*(M_2))$.
Moreover, for any $n\>= 1$, we have
\[ \Ext_{\alg}^n(i_*(M_2), i_*(M_1))
  \mathop{\cong}\limits^{\spadesuit} \Ext_{\alg}^n(M_2, i^!(i_*(M_1)))
  \mathop{\cong}\limits^{\clubsuit} \Ext_{\alg}^n(M_2, M_1)
  \mathop{=}\limits^{\heart} 0 \]
by using Lemmas \ref{Lem3.14} and \ref{Lem3.1} (see the isomorphism marked by $\spadesuit$),
Lemma \ref{2.6} (3) (see the isomorphism marked by $\clubsuit$),
and $(M_1, M_2)$ being an exceptional pair (see the equation marked by $\heart$).
It follows that \ref{E2'} holds for $(i_*(M_1), i_*(M_2))$.
Therefore, $(i_*(M_1), i_*(M_2))$ is an exceptional pairs in $\modcat(\alg)$.

One can prove that $(j_!(N_1), j_!(N_2))$ is an exceptional pair in $\modcat(\alg)$ by a similar way.
\end{proof}

\begin{theorem} \label{Th3.18}
Let $(X_1, \ldots, X_s)$ and $(Y_1, \ldots, Y_t)$ be two exceptional sequences
in $\modcat(\alg')$ and $\modcat(\alg'')$, respectively.
If $i^*$ and $i^!$ are exact functors, then $(i_*(X_1), \ldots, i_*(X_s))$
and $(j_!(Y_1), \ldots, j_!(Y_t))$ are exceptional sequences in $\modcat(\alg)$.
\end{theorem}

\begin{proof}
All modules $X_1$, $\ldots$, $X_s$ are exceptional in $\modcat(\alg')$ by the definition of exceptional sequence,
then Lemma \ref{Lem3.16} admits that $i_*(X_1)$, $\ldots$, $i_*(X_s)$ are exceptional in $\modcat(\alg)$.
Moreover, for each $1\=< i<j \=< s$, we have that $(X_j, X_i)$ is an exceptional pair in $\modcat(\alg')$
then, by the definition of exceptional sequence, $(i_*(X_j), i_*(X_i))$ is also
an exceptional pair in $\modcat(\alg)$ by using Lemma \ref{Lem3.17}.
Thus, $(i_*(X_1), \ldots, i_*(X_s))$ is an exceptional sequence in $\modcat(\alg)$.

One can prove that $(j_!(Y_1), \ldots, j_!(Y_t))$ is an exceptional sequence in $\modcat(\alg)$ by a similar way.
\end{proof}

\section{Examples} \label{Sect:Examp}

\subsection{Two examples for Corollary \ref{Cor3.2}}

First of all, we provide two examples for Corollary \ref{Cor3.2}, see Examples \ref{Ex3.3} and \ref{Ex3.4}.
In Example \ref{Ex3.3}, the instance illustrates that the condition ``${_A}R$ is projective'' is indispensable.
In Example \ref{Ex3.4}, we provide an instance for projective split extension of a finite-dimensional algebra $A$ by the nilpotent $(A,A)$-bimodule $Q$.

\begin{example} \label{Ex3.3} \rm
Let $A=\KK\Q_A/\I_A$ be a finite-dimensional algebra given by
$$ \Q_A = \xymatrix{
 1\ar[r]_{\alpha}    &     2\ar[r]_{\beta} & 3}
 ~\text{and}~
 \I_A = \langle \alpha\beta \rangle $$
and $R=\KK\Q_R/\I_R$ be a finite-dimensional algebra given by
$$ \Q_R = \xymatrix{
 1\ar[r]_{\alpha}    & 2\ar[r]_{\beta}  &3\ar@/_1.8pc/[ll]_{\gamma}}
 ~\text{and}~ \I_R = \langle \alpha\beta, \beta \gamma, \gamma\alpha \rangle
$$
Then $Q = \langle \gamma \rangle$ is an $(A,A)$-bimodule given by the left $A$-action
\[ A \times Q \to Q, ~ (\wp, k\gamma) \mapsto \wp\gamma =
  \begin{cases}
    \gamma, & \text{if}~\wp=\e_3; \\
    0, & \text{if}~\wp~\text{is another path}
  \end{cases} \]
  \begin{center}
    (and $\e_v$ is the path of length zero corresponded by the vertex $v$)
  \end{center}
and the right $A$-action
\[ Q \times A \to Q, ~ (\gamma, \wp) \mapsto \gamma\wp =
  \begin{cases}
    \gamma, & \text{if}~\wp=\e_1; \\
    0, & \text{if}~\wp~\text{is another path}.
  \end{cases} \]
By $\Q = \KK\gamma \subseteq \rad R = \KK\alpha+\KK\beta+\KK\gamma$,
$R$ is a split extension of $A$ by the nilpotent bimodule $Q$.
% ========== 2025-5-10 16:38:49 ==========
%Let $P_{iA}(P_{iR}) $ (resp. $I_{iA}(I_{iR}))$ be the projective (resp. injective) module which is related to the vertex $i$ in $A\modcat$(resp. $R$-mod). It is easy to see that
%$$P_{1A}=\begin{smallmatrix}1\\2\end{smallmatrix} = P_{1R}=I_{2A} = I_{2R} , P_{2A} = P_{2R} = I_{3A} = I_{3R}= \begin{smallmatrix}2\\3\end{smallmatrix},\\
% P_{3A} = 3,$$
%$$ P_{3R} = I_{1R}= \begin{smallmatrix}3\\1\end{smallmatrix},  I_{1A} = 1,
%= \begin{smallmatrix}1\\2\end{smallmatrix}.$$
% ========================================
Consider the projective resolution
\[ \mathbf{P}(M)= \ \
0 \To{} (3)_A \To{} ({^2_3})_A \To{} ({^1_2})_A  \To{} (1)_A \To{}{} 0 \]
of $M = (1)_A \in \modcat(A)$.
By $R = \KK\e_1\oplus_{\KK} \KK\e_2 \oplus_{\KK} \KK\e_3 \oplus_{\KK}
  \KK\alpha \oplus_{\KK} \KK\beta \oplus_{\KK} \KK\gamma + \I_R$,
we have
\begin{align*}
  (1)_A \otimes_A R
\cong\ & (\KK\e_1 + \langle\alpha, \alpha\beta\rangle) \otimes_A R \\
\cong\ & \KK\e_1 + \langle\alpha,\alpha\beta,\beta\gamma,\gamma\alpha\rangle \\
\cong\ & (1)_R,
\\
  ({^1_2})_A \otimes_A R
\cong\ & (\KK\e_1 \oplus_{\KK} \KK\alpha + \langle\alpha\beta\rangle) \otimes_A R \\
\cong\ & \KK\e_1 \oplus_{\KK} \KK\alpha + \langle\alpha\beta, \beta\gamma,\gamma\alpha\rangle \\
\cong\ & ({^1_2})_R,
\\
  ({^2_3})_A \otimes_A R
\cong\ & (\KK\e_2 \oplus_{\KK} \KK\beta) \otimes_A R \\
\cong\ & \KK\e_2 \oplus_{\KK} \KK\beta + \langle \beta\gamma \rangle \\
\cong\ & ({^2_3})_R,
\\   (3)_A \otimes_A R
\cong\ & \KK\e_3 \otimes_A R \\
\cong\ & \KK\e_3 \oplus_{\KK} \KK\gamma + \langle\gamma\alpha\rangle \\
\cong\ & ({^3_1})_R,
\end{align*}
then the sequence obtained by applying $F := -\otimes_A R$ to the projective resolution of $M$ is
\[ \mathbf{P}(M)\otimes_A R = \ \
0 \To{} ({^3_1})_R \To{} ({^2_3})_R \To{} ({^1_2})_R \To{} (1)_R \To{} 0
\]
which is not a projective resolution of $F(M) = M\otimes_A R = (1)_R$ since it is not left exact
(to be precise, $R = A\oplus Q$, as a right $A$-module, is not projective).
Indeed, one can check that the projective resolution of $(1)_R$ is
\[ \mathbf{P}(M\otimes_A R) = \ \
\cdots \To{} ({^2_3})_R \To{} ({^1_2})_R
  \To{} ({^3_1})_R \To{} ({^2_3})_R \To{} ({^1_2})_R \To{} (1)_R \To{} 0
 \]
Applying $\Hom_R(-,M\otimes_A R)$ to $\mathbf{P}(M\otimes_A R)$, we get the following complex
\begin{align*}
 \Hom_R&(\mathbf{P}(M\otimes_A R), M\otimes_A R) = \\
 0 & \To{} \Hom_{R}((1)_R,(1)_R) \To{} \Hom_{R}(({^1_2})_R,(1)_R) \\
   & \To{} \Hom_{R}(({^2_3})_R,(1)_R) = 0 \To{} \Hom_{R}(({^3_1})_R,(1)_R) = 0 \To{} \Hom_{R}(({^1_3})_R,(1)_R) \\
   & \To{} \Hom_{R}(({^2_3})_R,(1)_R) = 0 \To{} \Hom_{R}(({^3_1})_R,(1)_R) = 0 \To{} \Hom_{R}(({^1_3})_R,(1)_R) \\
   & \To{} \cdots.
\end{align*}
It follows that
\[ \Ext_R^3(M\otimes_A R, M \otimes_A R) \cong \Hom_R(({^1_3})_R,(1)_R) \cong_{\KK} \KK \ne 0, \]
where ``$\cong_{\KK}$'' represents the isomorphism of $\KK$-vector spaces.

On the other hand, $G=\Hom_R(R,-)$ is the functor such that $(-\otimes_A R, \Hom_R(R,-))$ is an adjoint pair,
and we have $\Hom_R(R,M\otimes_AR) = (1)_A$.
Applying $\Hom_R(-,\Hom_R(R,M\otimes_A R)) = \Hom_R(-,(1)_A)$ to $\mathbf{P}(M)$, we have
\begin{align*}
 \Hom_R&(\mathbf{P}(M), \Hom_R(R,M\otimes_A R)) = \Hom_R(\mathbf{P}(M),(1)_A) =  \\
 0 & \To{} \Hom_{R}((1)_A,(1)_A) \To{} \Hom_{R}(({^1_2})_A,(1)_A) \\
   & \To{} \Hom_{R}(({^2_3})_A,(1)_A) = 0 \To{} \Hom_{R}((3)_A,(1)_A) = 0 \To{} 0 \\
   & \To{} 0 \To{} 0 \To{} 0 \To{} \cdots.
\end{align*}
It follows that
\begin{align*}
 \Ext_A^3(M,\Hom_R(R,M\otimes_AR)) \cong \Ext_A^3((1)_A, (1)_A) = 0.
\end{align*}
Then we have
\[ 0 \ne \Ext_R^3(M\otimes_A R, M \otimes_A R) \not\cong \Ext_A^3(M,\Hom_R(R,M\otimes_AR)) = 0 \]
in this example.
\end{example}

\begin{example} \label{Ex3.4} \rm
Let $A=\kk\Q_A/\I_A$ be the finite dimensional algebra given in Example \ref{Ex3.3}
and $R'=\kk\Q_R'/\I_R'$ be the finite dimensional algebra whose quiver $\Q_R'$ is the quiver $\Q_R$ of the algebra $R$ given in Example \ref{Ex3.3} and $\I_R' = \langle \alpha\beta \rangle$. Then the epimorphism
\begin{center}
  $h: R' \to R'/\langle \gamma \rangle \cong A$, $\wp+\I_R' \mapsto (\wp+\I_R') +\langle\gamma+\I_R'\rangle$

  (note that ``$+\I_R$'' can be ignore for simplicity)
\end{center}
is a split extension of $A$ by the nilpotent bimodule $Q=\Ker(h) \cong \langle \gamma\rangle \cong \KK\gamma$, where
$Q$ is an $(A,A)$-bimodule given by the left $A$-action
\[ A \times Q \to Q, ~ (\wp, \gamma) \mapsto \wp\gamma =
  \begin{cases}
    \wp\gamma, & \text{if}~\wp\in\{\e_3,\beta\}; \\
    0, & \text{if}~\wp~\text{is another path}
  \end{cases} \]
and the right $A$-action
\[ Q \times A \to Q, ~ (\gamma, \wp) \mapsto \gamma\wp =
  \begin{cases}
    \gamma\wp, & \text{if}~\wp\in\{\e_1,\alpha\}; \\
    0, & \text{if}~\wp~\text{is another path}.
  \end{cases} \]
Then the quiver representation of $Q$, as a left $A$-module, is
\begin{align*}
    (\e_vQ, \varphi_{\alpha})_{v\in(\Q_A)_0,\alpha\in(\Q_A)_1}
=\ & (\xymatrix{
     \e_1\cdot\KK\gamma \ar[r]^{\alpha}
   & \e_2\cdot\KK\gamma \ar[r]^{\beta}
   & \e_3\cdot\KK\gamma }) \\
 \cong\ &
   (\xymatrix{
     0 \ar[r]^{\alpha}
   & 0 \ar[r]^{\beta}
   & \KK }).
\end{align*}
Thus, ${_AQ} = {_A(3)}$. It follows that ${_AR} \cong {_A({^1_2})} \oplus {_A({^2_3})}\oplus {_A(3)}^{\oplus 2}$ is a projective left $A$-module.
%%and ${_A}R$ is indeed a projective module. Actually,
%%${_A}R \cong (\begin{smallmatrix}1\\2\end{smallmatrix})^{2}\oplus \begin{smallmatrix}2\\3\end{smallmatrix}\oplus 3$.
%
%In addition, for any finite dimensions algebra $A$, if $R$ is a split extension of $A$ by the nilpotent bimodule $Q$(where $Q$ is generated by one arrow), and there are no any extra relations expect inherit them from $A$, then ${_A}R$ is always a projective $A\modcat$.
\end{example}

\subsection{An examples for Theorem \ref{Th2.1}}

\begin{example} \rm
Let $R=\KK\Q$ be a finite-dimensional $\KK$-algebra given by
\[ \Q = \xymatrix{ 1 \ar[r]^{\alpha} & 2 \ar[r]^{\beta} & 3}. \]
Take $A = R/\langle \alpha \rangle$ and $Q = \langle \alpha \rangle$,
then $\xi: R \to A$ be a split extension of $A$ by the nilpotent bimodule $Q$.
It is easy to check that
\[
(1)_A\otimes_{A}R \cong
  \left(\begin{smallmatrix}1\\2\\3\end{smallmatrix}\right)_R, ~
(\begin{smallmatrix}2\\3\end{smallmatrix})_A \otimes_{A} R
\cong (\begin{smallmatrix}2\\3\end{smallmatrix})_R, ~
(3)_A \otimes_{A}R \cong (3)_R,~
\text{and~}
(2)_A \otimes_{A}R \cong (2)_R.\]
Moreover, we have
\[
(1)_A\otimes_{A} Q_A \cong
  \left({^2_3}\right)_A, ~
(\begin{smallmatrix}2\\3\end{smallmatrix})_A \otimes_{A} Q_A
=0~
(3)_A \otimes_{A} Q_A =0~
\text{and~}
(2)_A \otimes_{A} Q_A =0.
\]
Note that $ {_AQ} \cong {_A(1)}^{\oplus 2}$, which is a direct sum of two copies of an indecomposable projective left $A$-module corresponding to the vertex $1$ of the quiver of $A$, then we obtain that $_AR$, as a left $A$-module, is isomorphic to the projective module $A\oplus {_AQ}$.

One can check that there are 9 complete exceptional sequences (=CES for short) in $\modcat(A)$ and 16 CESs in $\modcat(R)$, see \cite{S01}.
It is easy to check that $\Ext_A^1(-, (1)_A\otimes_A Q) = 0$, then, by Theorem \ref{Th2.1}, we have five CESs of $\modcat(R)$ by applying $-\otimes_A R$ to the sequences (a), (d), (e), (f), (i) in Table \ref{fig:6}.
\def\bs{\begin{smallmatrix}}
\def\es{\end{smallmatrix}}
\def\seq{sequence~}
\begin{table}[htbp]\centering
 \begin{tabular}{lcl}
  \toprule
  & CESs of $\modcat(A)$
  &  $-\otimes{_A}R$
  \\ \hline \\
  \seq (a):
  & $(\bs 3   \es,
     \bs 2\\3 \es,
     \bs  1  \es)$
  &  $\color{red}{(
    \bs 3 \es,
    \bs2\\3\es,
    \bs1\\2\\3\es)}$
  \\ \\
  \seq (b):
  & $(\bs 3  \es,
      \bs  1  \es,
      \bs2\\3\es)$
  & $(\bs  3  \es,
      \bs1\\2\\3\es,
      \bs2\\3\es)$
  \\ \\
  \seq (c):
  & $(\bs  1  \es,
      \bs  3  \es,
      \bs2\\3\es)$
  & $(\bs1\\2\\3\es,
      \bs  3  \es,
      \bs2\\3\es)$
  \\ \\
  \seq (d):
  & $(\bs2\\3\es,
      \bs 2 \es,
      \bs 1 \es)$
  &  \color{red}{
      $(\bs2\\3\es,
        \bs  2  \es,
        \bs1\\2\\3\es)$
        }
  \\ \\
  \seq (e):
  & $(\bs2\\3\es,
      \bs  1  \es,
      \bs  2  \es)$
  & \color{red}{
    $(\bs2\\3\es,
      \bs1\\2\\3\es,
      \bs  2  \es)$ }
  \\ \\
   \seq (f):
   & $(\bs  1  \es,
       \bs2\\3\es,
       \bs 2 \es)$
   & \color{red}{
     $(\bs1\\2\\3\es,
       \bs2\\3\es,
       \bs 2 \es)$
    }
  \\ \\
  \seq (g):
  & $(\bs  1  \es,
      \bs 2 \es,
      \bs 3 \es)$
  & $(\bs1\\2\\3\es,
      \bs 2 \es,
      \bs 3 \es)$
  \\ \\
  \seq (h):
  & $(\bs 2 \es,
      \bs 1 \es,
      \bs 3 \es)$
  & $(\bs 2 \es,
      \bs1\\2\\3\es,
      \bs 3 \es)$
  \\ \\
  \seq (i):
   & $(\bs 2 \es,
      \bs 3 \es,
      \bs 1 \es)$
   & $\color{red}{(
   \bs 2 \es,
   \bs 3 \es,
   \bs1\\2\\3\es)}$
   \\ \\
  \bottomrule
 \end{tabular}
 \caption{The table of CESs.}
\label{fig:6}
\end{table}
\end{example}

\subsection*{Authors' contributions}

D. Liu, H. Gao and Y.-Z. Liu contributed equally to this work.

\section*{Conflicts of Interest}

\section*{Data Availability}

The authors declares that no data was used in this article.

\section*{Declarations}

The authors states that there is no Conflict of interest.

\section*{Funding}

\begin{itemize}
  \item
Dajun Liu is supported by
the National Natural Science Foundation of China (No. 12101003),
and the Natural Science Foundation of Anhui province (No. 2108085QA07);
  \item
Hanpeng Gao is supported by
the National Natural Science Foundation of China (No. 12301041);
  \item
Yu-Zhe Liu is supported by
the National Natural Science Foundation of China (Grant No. 12401042, 12171207),
Guizhou Provincial Basic Research Program (Natural Science) (Grant Nos. ZD[2025]085 and ZK[2024]YiBan066)
and Scientific Research Foundation of Guizhou University (Grant Nos. [2022]53, [2022]65).
\end{itemize}

\section*{Acknowledgements}

We are greatly indebted to Yongyun Qin for helpful suggestions. 

%=========================================================

% 按文章中出现顺序排列参考文献：
%   \bibliographystyle{elsarticle-num}  % 作者姓全拼、名缩写, 数字标签, 书本名正体
%   \bibliographystyle{unsrt} % 作者姓名全拼, 数字标签, 书本名斜体
% 按字母顺序排列参考文献：
%  \bibliographystyle{plain} % 作者姓名全拼, 数字标签, 书本名正体
%  \bibliographystyle{abbrv} % 作者姓全拼、名缩写, 数字标签, 书本名斜体
%  \bibliographystyle{siam}  % 作者姓全拼、名缩写 (姓名使用textsc格式), 数字标签, 书本名斜体
%  \bibliographystyle{alpha} % 作者姓名全拼, 作者标签, 书本名斜体
%  \bibliographystyle{apalike}  % 作者姓全拼、名缩写, 美国心理学学会期刊样式

% \bibliography{referLiu20250515}

\begin{thebibliography}{10}

\bibitem{AI2012}
T.~Aihara and O.~Iyama.
\newblock Silting mutation in triangulated categories.
\newblock {\em J. London Math. Soc.}, 85(3):633--668, 2012.
\newblock \href{https://doi.org/10.1112/jlms/jdr055}
  {DOI:10.1112/jlms/jdr055},.

\bibitem{Asai2020}
S.~Asai.
\newblock Semibricks.
\newblock {\em Int. Math. Res. Notices}, 2020(16):4993--5054, 2020.
\newblock \href{https://doi.org/10.1093/imrn/rny150}{DOI:10.1093/imrn/rny150}.

\bibitem{AM1998}
I.~Assem and N.~Marmaridis.
\newblock Tilting modules over split-by-nilpotent extensions.
\newblock {\em Commun. Algebra}, 26(5):1547--1555, 1998.
\newblock
  \href{https://doi.org/10.1080/00927879808826219}{DOI:10.1080/00927879808826219}.

\bibitem{ASS2006}
I.~Assem, D.~Simson, and A.~Skowro\'{n}ski.
\newblock {\em Elements of the Representation Theory of Associative Algebras,
  Volume 1 Techniques of Representation Theory}.
\newblock Cambridge University Press, Cambridge, 2006.
\newblock
  \href{https://doi.org/10.1017/CBO9780511614309}{DOI:10.1017/CBO9780511614309}.

\bibitem{AZ2003}
I.~Assem and D.~Zacharia.
\newblock On split-by-nilpotent extensions.
\newblock {\em Colloq. Math.}, 98(2):259--275, 2003.
\newblock \url{http://eudml.org/doc/285195}.

\bibitem{BBDG1982}
A.~A. Be\u{\i}linson, J.~Bernstein, P.~Deligne, and O.~Gabber.
\newblock Faisceaux pervers.
\newblock In {\em Analysis and topology on singular spaces, {I} ({L}uminy,
  1981)}, volume 100, pages 5--171. Soci\'{e}t\'{e} Math\'{e}matique de France,
  Paris, 1982.
\newblock
  \url{https://smf.emath.fr/sites/default/files/2019-05/smf_ast_100__sample.pdf}.

\bibitem{BM2021}
A.~B. Buan and B.~R. Marsh.
\newblock $\tau$-exceptional sequences.
\newblock {\em J. Algebra}, 585:36--68, 2021.
\newblock
  \href{https://doi.org/10.1016/j.jalgebra.2021.04.038}{DOI:10.1016/j.jalgebra.2021.04.038}.

\bibitem{Chen2012}
X.-W. Chen.
\newblock A recollement of vector bundles.
\newblock {\em Bull. Lond. Math. Soc.}, 44(2):271--284, 2012.
\newblock \href{https://doi.org/10.1112/blms/bdr092}{DOI:10.1112/blms/bdr092}.

\bibitem{CB1993}
W.~Crawley-Boevey.
\newblock Exceptional sequences of representations of quivers.
\newblock In {\em Representations of Algebras: Sixth International Conference
  (August 19--22, 1992, Ottawa, Ontario, Canada)}, volume~7, pages 117--124,
  Ottawa, 1993. American Mathematical Society.
\newblock
  \href{https://mathscinet.ams.org/mathscinet/article?mr=1206935}{MR:1206935}.

\bibitem{FP2004}
V.~Franjou and T.~Pirashvili.
\newblock Comparison of {A}belian categories recollements.
\newblock {\em Doc. Math.}, 9:41--56, 2004.
\newblock
  \href{https://ems.press/content/serial-article-files/25938}{zbMath:2095716}.

\bibitem{Gao2021}
H.~Gao.
\newblock Semibricks over split-by-nilpotent extensions.
\newblock {\em B. Korean Math. Soc.}, 1:183--193, 2021.
\newblock
  \href{https://doi.org/10.4134/BKMS.b200189}{DOI:10.4134/BKMS.b200189}.

\bibitem{GR1987}
A.~L. Gorodentsev and A.~N. Rudakov.
\newblock Exceptional vector bundles on projective spaces.
\newblock {\em Duke Math. J.}, 54(1):115--130, 1987.
\newblock \href{http://dx.doi.org/10.1215/S0012-7094-87-05409-3} {DOI:
  10.1215/S0012-7094-87-05409-3}.

\bibitem{LW2022}
D.~Liu and J.~Wei.
\newblock Tilting pair over split-by-nilpotent extensions.
\newblock {\em J. Algebra Appl.}, 21(1):paper no. 2250019, 2022.
\newblock
  \href{http://dx.doi.org/10.1142/S0219498822500190}{DOI:10.1142/S0219498822500190}.

\bibitem{M2018}
X.~Ma.
\newblock On gluing tilting modules (in {C}hinese).
\newblock {\em Sci. Sin. Math.}, 48(11):1729--1738, 2018.
\newblock
  \href{https://doi.org/10.1360/N012017-00266}{DOI:10.1360/N012017-00266}.

\bibitem{MXZ2022}
X.~Ma, Z.~Xie, and T.~Zhao.
\newblock Supprt $\tau$-tilting modules and recollements.
\newblock {\em Colloq. Math.}, 167(2):303--328, 2022.
\newblock
  \href{https://doi.org/10.4064/cm8358-11-2020}{DOI:10.4064/cm8358-11-2020}.

\bibitem{MZ2020}
X.~Ma and T.~Zhao.
\newblock Recollements and tilting modules.
\newblock {\em Comm. Algebra}, 48(12):5163--5175, 2020.
\newblock
  \href{https://doi.org/10.1080/00927872.2020.1781874}{DOI:10.1080/00927872.2020.1781874}.

\bibitem{Sua2020}
S.~Pamela.
\newblock Split-by-nilpotent extensions and support $\tau$-tilting modules.
\newblock {\em Algebr. Represent. Theory}, 23:2295--2313, 2020.
\newblock \href{https://doi.org/10.1007/s10468-019-09932-1}
  {10.1007/s10468-019-09932-1},.

\bibitem{PV2019}
C.~E. Parra and J.~Vit{\'o}ria.
\newblock Properties of {A}belian categories via recollements.
\newblock {\em J. Pure Appl. Algebra}, 223(9):3941--3963, 2019.
\newblock
  \href{https://doi.org/10.1016/j.jpaa.2018.12.013}{DOI:10.1016/j.jpaa.2018.12.013}.

\bibitem{Psa2014}
C.~Psaroudakis.
\newblock Homological theory of recollements of {A}belian categories.
\newblock {\em J. Algebra}, 398(15):63--110, 2014.
\newblock \href{http://dx.doi.org/10.1016/j.jalgebra.2013.09.020}
  {DOI:10.1016/j.jalgebra.2013.09.020}.

\bibitem{QinHan2016}
Y.~Qin and Y.~Han.
\newblock Reducing homological conjectures by {{\(n\)}}-recollements.
\newblock {\em Algebr. Represent. Theory}, 19(2):377--395, 2016.
\newblock
  \href{https://doi.org/10.1007/s10468-015-9578-z}{DOI:10.1007/s10468-015-9578-z}.

\bibitem{R1994}
C.~M. Ringel.
\newblock The braid group action on the set of exceptional sequences of a
  hereditary artin algebra.
\newblock In {\em Abelian Group Theory and Related Topics}, volume 171, pages
  117--124, Princeton, 1994. Contemporary Mathematics, Universal Wiser
  Publisher.
\newblock
  \href{https://doi.org/10.1090/conm/171/01761}{DOI:10.1090/conm/171/01761}.

\bibitem{S01}
U.~Seidel.
\newblock Exceptional sequences for quivers of Dynkin type.
\newblock {\em Commun. Algebra}, 29(3), 1373--1386, 2001.
\newblock
  \href{https://doi.org/10.1081/AGB-100001691}{DOI:10.1081/AGB-100001691}.

\bibitem{S2021}
E.~Sen.
\newblock Weak exceptional sequences.
\newblock {\em Quaest. Math.}, 44(9):1155--1171, 2021.
\newblock \href{https://doi.org/10.2989/16073606.2020.1777483}
  {DOI:10.2989/16073606.2020.1777483}.

\bibitem{WT2022}
E.~P. Westin and M.~Thuresson.
\newblock Tilting modules and exceptional sequences for a family of dual
  extension algebras.
\newblock {\em Algebr. Represent. Theory}, 26(5):1549--1581, 2022.
\newblock \href{https://doi.org/10.1007/s10468-022-10142-5}{DOI:
  10.1007/s10468-022-10142-5}.

\bibitem{Zhang2017}
C.~Zhang.
\newblock On the global cohomological width of {A}rtin algebras.
\newblock {\em Colloq. Math.}, 146:31--46, 2017.
\newblock
  \href{https://doi.org/10.4064/cm6714-10-2015}{DOI:10.4064/cm6714-10-2015}.

\bibitem{zhang2017subcate}
C.~Zhang and H.~Cai.
\newblock A note on thick subcategories and wide subcategories.
\newblock {\em Homology, Homotopy and Applications}, 19(2):131--139, 2017.
\newblock
  \href{https://doi.org/10.4310/HHA.2017.v19.n2.a8}{DOI:10.4310/HHA.2017.v19.n2.a8}.

\bibitem{ZZ2020}
H.~Zhang and X.~Zhu.
\newblock {G}orenstein global dimension of recollements of abelian categories.
\newblock {\em Commun. Algebra}, 48(2):467--483, 2020.
\newblock
  \href{https://doi.org/10.1080/00927872.2019.1648650}{DOI:10.1080/00927872.2019.1648650}.

\bibitem{ZZ2020JAA}
H.~Zhang and X.~Zhu.
\newblock Resolving resolution dimension of recollements of abelian categories.
\newblock {\em J. Algebra Appl.}, 20(10):no.2150179, 2021.
\newblock
  \href{https://doi.org/10.1142/S0219498821501796}{DOI:10.1142/S0219498821501796}.

\bibitem{Zhang2013Gorenstein}
P.~Zhang.
\newblock Gorenstein-projective modules and symmetric recollements.
\newblock {\em J. Algebra}, 388:65--80, 2013.
\newblock
  \href{https://doi.org/10.1016/j.jalgebra.2013.05.008}{DOI:10.1016/j.jalgebra.2013.05.008}.

\bibitem{ZhangY2023}
Y.~Zhang.
\newblock Reduction of wide subcategories and recollements.
\newblock {\em Algebra Colloq.}, 30(4):713--720, 2023.
\newblock
  \href{https://doi.org/10.1142/S1005386723000536}{DOI:10.1142/S1005386723000536}.



\end{thebibliography}

% \listofchanges

\def\cprime{$'$}

\end{document}